\documentclass{amsart}
\usepackage{latexsym,amsfonts,amssymb,amsmath,amsthm,cite}

\newtheorem{theorem}{Theorem}[section]
\newtheorem{lemma}[theorem]{Lemma}
\newtheorem{proposition}[theorem]{Proposition}

\theoremstyle{definition}

\theoremstyle{remark}
\newtheorem{remark}[theorem]{Remark}

\numberwithin{equation}{section}

\newcommand\norm[1]{\lVert#1\rVert}

\begin{document}

\title[Comparable Almost Periodic Reaction-Diffusion
Systems] {Convergence in Comparable Almost Periodic
Reaction-Diffusion Systems with Dirichlet Boundary Condition}

\author{Feng Cao}
\address{Department of Mathematics, Nanjing University of Aeronautics and Astronautics,
Nanjing, Jiangsu 210016, P. R. China}
\email{fcao@nuaa.edu.cn}
\thanks{The first author is supported by NSF of China No. 11201226, the New Teachers' Fund for Doctor Stations No. 20123218120032,
and the Fundamental Research Funds for the Central Universities NO.
NS2012001.}

\author{Yelai Fu}
\address{Department of Mathematics, Nanjing University of Aeronautics and Astronautics,
Nanjing, Jiangsu 210016, P. R. China}
 \email{fuyelai@126.com}

%\thanks{}

%    General info

%\date{August 13, 2013.}



\begin{abstract}
The paper is to study the asymptotic dynamics in nonmonotone
comparable almost periodic reaction-diffusion system with Dirichlet
boundary condition, which is comparable with uniformly stable
strongly order-preserving system. By appealing to the theory of
skew-product semiflows, we obtain the asymptotic almost periodicity
of uniformly stable solutions to the comparable reaction-diffusion
system.
\end{abstract}

\maketitle

%\section*{}
%This is an example of an unnumbered first-level heading.

%% The correct journal style for \specialsection is all uppercase; a known bug
%% in amsart.cls prevents this, so input must be uppercase until it is fixed.
%\specialsection*{This is a Special Section Head}
%\specialsection*{THIS IS A SPECIAL SECTION HEAD}
%This is an example of a special section head%
%%%%%%%%%%%%%%%%%%%%%%%%%%%%%%%%%%%%%%%%%%%%%%%%%%%%%%%%%%%%%%%%%%%%%%%%
%\footnote{Here is an example of a footnote. Notice that this footnote
%text is running on so that it can stand as an example of how a footnote
%with separate paragraphs should be written.
%\par
%And here is the beginning of the second paragraph.}%
%%%%%%%%%%%%%%%%%%%%%%%%%%%%%%%%%%%%%%%%%%%%%%%%%%%%%%%%%%%%%%%%%%%%%%%%
%.

\section{Introduction}
In the last 50 years or so, many of the concepts of dynamical
systems have been applied to the study of partial different
equations (see \cite{chen,chen2,chen3,chow,chow2,Hale,Hen,ShenYi,S},
etc.). In this paper, we shall study the long-term behaviour of the
solutions of some non-autonomous comparable reaction-diffusion
equations.

We consider the almost periodic reaction-diffusion system with
Dirichlet boundary condition:

\begin{equation}\label{1.1}
\left\{\begin{array}{ll} \dfrac{\partial v_i}{\partial t}=
d_i(t)\Delta v_i +F_i(t,v_1,\cdots,v_n),\qquad x\in \Omega,
\, t>0,\\
v_i(t,x)=\,0, \qquad \qquad\qquad \qquad \qquad x\in \partial\Omega,
\, t>0,
\\
v_i(0,x)=v_{0,i}(x), \qquad \qquad\qquad \qquad x\in \bar{\Omega},\,
1\leqq i\leqq n,
\end{array}
 \right.
\end{equation}
where $\Omega$ is a bounded domain in $\mathbb{R}^n$ with smooth
boundary. $d=(d_1(\cdot),\cdots,d_n(\cdot))\in C(\mathbb {R},\mathbb
{R}^n)$ is assumed to be an almost periodic vector-valued function
bounded below by a positive real vector. The nonlinearity
 $F=(F_1,\cdots,F_n): \mathbb{R}\times
\mathbb{R}^n\to \mathbb{R}^n$ is $C^1$-admissible and uniformly
almost periodic in $t$, and $F$ points into $\mathbb {R}_+^n$ along
the boundary of $\mathbb {R}_+^n$: $F_i(t,v)\geqq 0 \textnormal{
whenever }
 v\in \mathbb {R}_+^n \textnormal{ with
} v_i=0 \textnormal{ and } t\in \mathbb{R}^+$. However, {\it $F$ has
no monotonicity properties}.

In order to study properties of the solutions of such a non-monotone
equation, an effective approach is to exhibit and utilize certain
comparison techniques (see \cite{ConSm,Bro1,Bro2,Sm}). As pointed
out in \cite[Section 4]{Smi3}, the comparison technique involves
monotone systems in a natural way: the original non-monotone systems
are comparable with certain monotone ones. Thus, we assume that
there exists a function $f:\mathbb{R}\times \mathbb{R}^n_+\to
\mathbb{R}^n$ with $f(t,v)\geqq F(t,v)$ (or $f(t,v)\leqq F(t,v)$),
$\forall(t,v)\in \mathbb{R}\times \mathbb{R}_+^n$. Also, we assume
that $f$ satisfies (H1)-(H4) in section 2. Then we get a strongly
order-preserving system (see section 2 for details):
\begin{equation}\label{1.2}
\left\{\begin{array}{ll} \dfrac{\partial u_i}{\partial t}=
d_i(t)\Delta u_i +f_i(t,u_1,\cdots,u_n),\qquad x\in \Omega,
\, t>0,\\
u_i(t,x)=\,0, \qquad \qquad\qquad \qquad\qquad x\in \partial\Omega,
\, t>0,
\\
u_i(0,x)=u_{0,i}(x), \qquad \qquad \qquad\qquad x\in \bar{\Omega},\,
1\leqq i\leqq n.
 \end{array}
 \right.
\end{equation}

We want to know whether such a non-monotone system (\ref{1.1})
inherits certain asymptotic behaviour from its strongly
order-preserving partner (\ref{1.2}). Note that a unified framework
to study nonautonomous equations is based on the so-called
skew-product semiflows  (see \cite{Sell,ShenYi}). Since even the
strongly monotone (which is a stronger notion than strongly
order-preserving) skew-product semiflows can possess very
complicated chaotic attractors (see \cite{ShenYi}), we hence assume
that the strongly order-preserving partner is `uniformly stable',
and to establish the asymptotic 1-cover property of the
corresponding  strongly order-preserving skew-product semiflow.

As far as we know, there are only a few works on the related topics.
Jiang \cite{Jiang3} proved the global convergence of the comparable
discrete-time or continuous-time system provided that all the
equilibria of its monotone partner form a totally ordered curve.
Recently, Cao, Gyllenberg and Wang\cite{CGW} established the
asymptotic 1-cover property of the comparable skew-product
semiflows, whose partner systems are eventually strongly monotone
and uniformly stable. Here we emphasize that for reaction-diffusion
system with Dirichlet boundary condition, the cone $X_+$ has empty
interior in the state space $X=\Pi_1^n C_0(\bar{\Omega})$ (see
section 2 for details). Thus, the skew-product semiflow  generated
by its partner is only strongly order-preserving, but not eventually
strongly monotone (see \cite[Chapter 6]{HirSmi}). So we have to find
another way to get the corresponding asymptotic dynamics for
Dirichlet problem.

Motivated by \cite{JZH}, in order to get the asymptotic behavior of
solutions to comparable almost periodic  reaction-diffusion system
(\ref{1.1}), we first prove that every precompact trajectory of the
strongly order-preserving system (\ref{1.2}) is asymptotic to a
1-cover of the base flow (see Proposition \ref{theorem 3.3}). Based
on this, for the uniformly stable and strongly order-preserving
skew-product semiflow generated by (\ref{1.2}), we can get the
topological structure of the set of the union of all 1-covers
similarly as \cite{CGW} (see Lemma \ref{lemma 4.1}). With such
tools, we are able to establish the 1-covering property of uniformly
stable omega-limit sets of comparable skew-product semiflow (see
Proposition \ref{prop 4.3}), and thus obtain the asymptotic almost
periodicity of uniformly stable solutions to system (\ref{1.1}).

The paper is organized as follows. In section 2, we present some
basic definitions and our main result. In Section 3 we prove the
main result.

\section{Definitions and the main result}

A subset $S$ of $\mathbb{R}$ is said to be {\it relatively dense} if
there exists $l>0$ such that every interval of length $l$ intersects
$S$. A function $f$, defined and continuous on $\mathbb{R}$, is {\it
almost periodic} if, for any $\varepsilon>0$, the set
$T(f,\varepsilon)=\{s\in
\mathbb{R}:|f(t+s)-f(t)|<\varepsilon,\,\forall t\in \mathbb{R}\}$ is
relatively dense. A continuous function $f : \mathbb{R}\times
\mathbb{R}^m \mapsto \mathbb{R}^n$ is said to be {\it admissible}
if, for every compact subset $K \subset \mathbb{R}^m$, $f$ is
bounded and uniformly continuous on $\mathbb{R}\times K$. Besides,
if $f$ is of class $C^r (r \geqq 1)$ in $x \in \mathbb{R}^m$, and
$f$ and all its partial derivatives with respect to $x$ up to order
$r$ are admissible, then we say that $f$ is $C^r$-{\it admissible}.
A function $f \in C(\mathbb{R}\times \mathbb{R}^m,\mathbb{R}^n)$ is
{\it uniformly almost periodic in} $t$, if $f$ is both admissible
and almost periodic in $t\in \mathbb{R}$.

Let $f \in C(\mathbb{R}\times \mathbb{R}^m,\mathbb{R}^n)$ be
uniformly almost periodic, one can define the Fourier series of $f$
(see \cite{ShenYi,Ve}), and the {\it frequency module}
$\mathcal{M}(f)$ of $f$ as the smallest Abelian group containing a
Fourier spectrum. Let $f,g\in C(\mathbb{R}\times
\mathbb{R}^m,\mathbb{R}^n)$ be two uniformly almost periodic
functions in $t$. One has $\mathcal{M}(f)=\mathcal{M}(g)$ if and
only if the flow $(H(g),\mathbb{R})$ is isomorphic to the flow
$(H(f),\mathbb{R})$ (see, \cite{Fi} or \cite[Section
1.3.4]{ShenYi}). Here $H(f)={\rm cl}\{f\cdot\tau:\tau\in
\mathbb{R}\}$ is called the {\it hull of $f$}, where
$f\cdot\tau(t,\cdot)=f(t+\tau,\cdot)$ and the closure is taken under
the compact open topology.

Let $(Y,d_Y)$ be a compact metric space with metric $d_Y$. A
\emph{continuous flow} $\sigma: \mathbb{R}\times Y \rightarrow Y$,
$(t,y) \rightarrow \sigma{(t,y)}=\sigma_t(y)=y\cdot t$ is called
\emph{minimal} if $Y$ has no other nonempty compact invariant subset
but itself. Here  a subset $Y_1 \subset Y$ is \emph{invariant} if
$\sigma_{t}(Y_1) = Y_1$ for every $t \in \mathbb{R}$.

Consider the almost periodic reaction-diffusion system with
Dirichlet boundary condition:
\begin{equation}\label{IBVP-sys}
\left\{\begin{array}{ll} \dfrac{\partial v_i}{\partial t}=
d_i(t)\Delta v_i +F_i(t,v_1,\cdots,v_n),\qquad x\in \Omega,
\, t>0,\\
v_i(t,x)=\,0, \qquad \qquad\qquad \qquad \qquad x\in \partial\Omega,
\, t>0,
\\
v_i(0,x)=v_{0,i}(x), \qquad \qquad\qquad \qquad x\in \bar{\Omega},\,
1\leqq i\leqq n,
 \end{array}
 \right.
\end{equation}
where $\Omega$ is a bounded domain in $\mathbb{R}^n$ with smooth
boundary. $\Delta$ is the Laplacian operator on $\mathbb{R}^n$.

Let $d=(d_1(\cdot),\cdots,d_n(\cdot))\in C(\mathbb{R},\mathbb{R}^n)$
be an almost periodic vector-valued function and for some $d_0 > 0$,
$d_i(t) \geqq d_0$, $\forall t \in \mathbb{R}$, $1 \leqq i \leqq n$.
The nonlinearity $F=(F_1,\cdots,F_n): \mathbb{R}\times
 \mathbb{R}^n\to \mathbb{R}^n$ is
$C^1$-admissible and uniformly almost periodic in $t$. Let
$v=(v_1,\cdots,v_n)$, we also assume that
$$\textnormal{ (I) } \qquad F_i(t,v)\geqq 0 \textnormal{ whenever }
 v\in \mathbb{R}_+^n \textnormal{ with
} v_i=0 \textnormal{ and }  t\in \mathbb{R}^+.\qquad$$

Denote $X=\Pi_1^n C_0(\bar{\Omega})$ ($C_0(\bar{\Omega}):=\{\phi\in
C(\bar{\Omega},\mathbb{R}):\phi|_{\partial\Omega}=0\}$) and the
standard cone $X_+=\{u\in X:u(x)\in
\mathbb{R}_+^n,x\in\bar{\Omega}\}$. Then the cone $X_+$ induces an
\emph{ordering} on $X$ via $x_{1}\leqq x_{2}$ if $x_{2}- x_{1}\in
X_+$. We write $x_{1}< x_{2}$ if $x_{2}- x_{1}\in X_+\setminus
\{0\}$.  Let $x\in X$ and a subset $U\subset X$. We write $x<_r U$
if $x<_r u$ for all $u\in U$. Given two subsets $A$, $B\subset X$,
we write $A<_r B$ if $a<_r b$ holds for each choice of $a\in A$,
$b\in B$. Here $<_r$ represents $\leqq$ or $<$. $x>_r U$ is
similarly defined. Obviously, every compact subset in $X$ has both a
greatest lower bound and a least upper bound.

Let $H(d,F)$ be the hull of the function $(d,F)$. Then the time
translation $(\mu,G)\cdot t$ of $(\mu,G)\in H(d,F)$ induces a
 compact and minimal flow on $H(d,F)$ (see \cite{Sell} or \cite{ShenYi}).
By the standard theory of reaction-diffusion systems (see
\cite[Chapter 6]{HirSmi}), it follows that for every $v_0\in X_+$
and $(\mu,G)\in H(d,F)$, the system
\begin{equation}\label{IBVP-sys-g}
\left\{\begin{array}{ll} \dfrac{\partial v_i}{\partial t}=
\mu_i(t)\Delta v_i +G_i(t,v),\qquad x\in \Omega,
\, t>0,\\
v_i(t,x)=\,0, \qquad \qquad\qquad \qquad x\in \partial\Omega, \,
t>0,
\\
v(0,x)=v_{0}(x), \qquad \qquad \qquad x\in \bar{\Omega},\, 1\leqq
i\leqq n
 \end{array}
 \right.
\end{equation}
admits a (locally) unique regular solution $v(t,\cdot,v_0;\mu,G)$ in
$X_+$. This solution also continuously depends on $(\mu,G)\in
H(d,F)$ and $v_0\in X_+$ (see \cite{Hen}). Thus, \eqref{IBVP-sys-g}
induces a (local) skew-product semiflow $\Gamma$ on $X_+\times
H(d,F)$ with
$$\Gamma_t(v_0,(\mu,G))=(v(t,\cdot,v_0;\mu,G),(\mu,G)\cdot t),\quad \forall\,
(v_0,(\mu,G))\in X_+\times H(d,F), \,t\geqq 0.$$

Now we assume that there exists a function $f\in
C^1(\mathbb{R}\times\mathbb{R}^n_+,\mathbb{R}^n)$, which is
$C^1$-admissible and uniformly almost periodic in $t$, satisfying

\begin{enumerate}
\item [ (H1)]$$ f(t,v)\geqq F(t,v) \qquad \textnormal{ for all } (t,v)\in
\mathbb{R}\times \mathbb{R}_+^n. $$ with its frequency module
$\mathcal{M}(f)=\mathcal{M}(F)$ (thus $H(d,f)\cong H(d,F)$);

\item [ (H2)]  $f_i(t,0)=0 \,(1\leqq i\leqq n)$;

\item [ (H3)] $\frac{\partial f_i}{\partial x_j}(t,x)\geqq 0$ for all
$1\leqq i\neq j\leqq n$, and there is a $\delta>0$ such that if two
nonempty subsets $I,J$ of $\{1,2,\cdots,n\}$ form a partition of
$\{1,2,\cdots,n\}$, then for any $(t,x)\in
\mathbb{R}\times\mathbb{R}^n_+$, there exist $i\in I$, $j\in J$ such
that $|\dfrac{\partial f_i}{\partial x_j}(t,x)|\geqq \delta>0$;

\item [ (H4)] Every nonnegative solution of ordinary differential
system $\dot{u} = g(t,u), g \in H(f)$, is bounded.

\end{enumerate}

It is easy to see that, for any $(\mu,G)\in H(d,F)$, there exists a
$(\mu,g)\in H(d,f)$ such that
$$g(t,v)\geqq G(t,v) \textnormal{ for all } (t,v)\in
\mathbb{R}\times \mathbb{R}_+^n.$$ Denote $Y=H(d,f)$. Then we can
consider the following new reaction-diffusion system:
\begin{equation}\label{IBVP-sys-g-1}
\left\{\begin{array}{ll} \dfrac{\partial u_i}{\partial t}=
\mu_i(t)\Delta u_i +g_i(t,u),\qquad x\in \Omega,
\, t>0,\\
u_i(t,x)=\,0, \qquad \qquad\qquad \qquad x\in \partial\Omega, \,
t>0,
\\
u(0,x)=u_{0}(x)\in X_+,  \qquad \qquad x\in \bar{\Omega},\, 1\leqq
i\leqq n,
 \end{array}
 \right.
\end{equation}
which induces the following global skew-product semiflow:
\begin{equation}\label{equ6} \Pi_t:X_+\times Y\rightarrow X_+\times Y;
~(u_0,y=(\mu,g))\mapsto (u(t,\cdot,u_0,y),y\cdot t),~ t\in
\mathbb{R}^+,\end{equation} where $u(t,\cdot,u_0,y)$ is the unique
regular global solution of (\ref{IBVP-sys-g-1}) in $X_+$. Without
any confusion, we also write $u(t,\cdot,u_0,y)$ as $u(t,u_0,y)$.

Clearly, by the comparison principle and (H4),
%A subset $A\subset X\times Y$ is \emph{positively invariant} if
%$\Pi_{t}(A)\subset A$ for all $t\in \mathbb{R}^{+}$; and {\it
%totally invariant} if $\Pi_{t}(A)=A$ for all $t\in \mathbb{R}^{+}$.
the forward orbit $O^+(x,y)= \left\{\Pi_{t}(x,y) : t\geqq 0
\right\}$ of any $(x,y)\in X_+\times Y$ is precompact. Thus the
omega-limit set of $(x,y)$, defined by
$\omega(x,y)=\{(\hat{x},\hat{y}) \in X_+\times Y :
\Pi_{t_{n}}(x,y)\rightarrow (\hat{x},\hat{y}) (n\rightarrow\infty)
\text{\ for some sequence\ } t_{n}\rightarrow \infty \}$, is a
nonempty, compact and invariant subset in $X_+\times Y$. A forward
orbit $O^+(x_0,y_0)$ of $\Pi_t$ is said to be {\it uniformly stable}
if for every $\varepsilon>0$ there is a
$\delta=\delta(\varepsilon)>0$, called the {\it modulus of uniform
stability},  such that for every $x\in X_+$, if $s\geqq 0$ and
$\norm{u(s,x_0,y_0)-u(s,x,y_0)}\leqq \delta(\varepsilon)$ then
$$\norm{u(t+s,x_0,y_0)-u(t+s,x,y_0)}<\varepsilon \textnormal{ for
each }t\geqq 0.$$ Here we assume that every forward orbit of $\Pi_t$
in (\ref{equ6}) is uniformly stable, which can be guaranteed by the
existence of invariant functional.

Let $P:X_+\times Y \to Y$ be the natural projection. A compact
positively invariant set $K\subset X_+\times Y$ is called a {\it
$1$-cover} of $Y$ if $P^{-1}(y)\cap K$ contains a unique element for
every $y\in Y$. If we write the 1-cover $K=\{(c(y),y):y\in Y\}$,
then $c:Y\rightarrow X$ is continuous with $\Pi_t(c(y),y)=(c(y\cdot
t),y\cdot t)$, $\forall t\geqq0$. For the sake of brevity, we
hereafter also write $c(\cdot)$ as a {\it $1$-cover} of $Y$.

For skew-product semiflows, we always use the order relation on each
fiber $P^{-1}(y)$, and write $(x_1,y)\leqq (<)\, (x_2,y)$ if
$x_1\leqq x_2$ ($x_1<x_2$). Recall that the skew-product semiflow
$\Pi_t$ is called {\it monotone} if
$$\Pi_{t}(x_1,y)\leqq \Pi_{t}(x_2,y)$$ whenever $(x_1,y)\leqq (x_2,y)$
and $t\geqq0$. Moreover, $\Pi_t$ is {\it strongly order-preserving}
if it is monotone and there is a $t_0>0$ such that, whenever
$(x_1,y)<(x_2,y)$ there exist open subsets $U$, $V$ of $X_+$ with
$x_1\in U$,  $x_2\in V$ satisfying
$$\Pi_{t}(U,y) < \Pi_{t}(V,y) \quad \text{for all }
t\geqq t_0.$$ $\Pi_t$ is called {\it fiber-compact} if there exists
a $\bar{t}>0$ such that, for any $y\in Y$ and bounded subset
$B\subset X$, $\Pi_t(B,y)$ has compact closure in $P^{-1}(y\cdot t)$
for every $t>\bar{t}$. Then according to (H3), \cite[Chapter
6]{HirSmi} and \cite[Section 6]{JZH}, one can obtain that $\Pi_t$ in
(\ref{equ6}) is strongly order-preserving and fibre-compact.

By (H1), similarly as the proof of Lemma 5.2 in \cite{CGW}, we can
get that $\Gamma_t$ is upper-comparable with respect to $\Pi_t$ in
the sense that if $\Gamma_t(x_1,y)\leqq\Pi_t(x_2,y)$ whenever
$(x_1,y),(x_2,y)\in X_+\times Y$ with $(x_1,y)\leqq(x_2,y)$.

Now we are in a position to state our main result.

\begin{theorem}\label{theorem 2.1}
Any uniformly stable $L^\infty$-bounded solution of \eqref{IBVP-sys}
is asymptotic to an almost periodic solution.
\end{theorem}

\begin{remark}
We note that for reaction-diffusion system with Dirichlet boundary
condition (\ref{IBVP-sys}), the cone $X_+$ has empty interior in the
state space $X=\Pi_1^n C_0(\bar{\Omega})$. Thus, the skew-product
semiflow  generated by its monotone partner (\ref{IBVP-sys-g-1}) is
only strongly order-preserving, but not eventually strongly
monotone. Consequently, the results in \cite{CGW} can't be used to
study the asymptotic behavior of the solutions to system
(\ref{IBVP-sys}).
\end{remark}

\section{Proof of Theorem \ref{theorem 2.1}}

In order to get the asymptotic almost periodicity of solutions to
system (\ref{IBVP-sys}), we first investigate the asymptotic
behavior of its strongly order-preserving partner. Motivated by
\cite{JZH}, we establish the 1-cover property of omega limit sets
for the strongly order-preserving and uniformly stable skew-product
semiflows $\Pi_t$.

The following result is adopted from \cite[P. 19]{RJGR} or \cite[P.
29]{ShenYi}, see also \cite[P. 634]{NOS}.

\vskip 3mm

\begin{theorem}\label{theorem 3.0}
Let $\Theta_t$ be a skew-product semiflow on $X_+\times Y$. If a
forward orbit $O^+_\Theta(x_0,y_0)$ of $\Theta_t$ is precompact and
uniformly stable, then its omega-limit set $\omega_\Theta(x_0,y_0)$
admits a flow extension which is minimal.
\end{theorem}

Now fix $(x_0,y_0)\in X_+\times Y$ and let $K=\omega(x_0,y_0)$ be
its omega-limit set with respect to $\Pi_t$. For any given $y\in Y$,
we define
$$(p(y),y) = \mbox{g.l.b. of }~K \bigcap P^{-1}(y).$$
Then from \cite[Proposition 3.1]{JZH}, it follows that
$\omega(p(y),y)$  is $1$-cover of $Y$. Denote
$\{(p_{\ast}(y),y)\}=\omega(p(y),y)\bigcap P^{-1}(y) $,
 by \cite[Proposition 3.2]{JZH} one has
\begin{equation}\label{3.0}u(t,p_\ast(y),y)=p_\ast(y\cdot t)~\mbox{ for any }y\in Y \mbox{
and } t\in\mathbb{R}.\end{equation} So we can denote the 1-cover
$\omega(p(y),y)$ by $p_\ast(\cdot)$.

\vskip 3mm

\begin{lemma}
\label{lemma 3.1} Assume that there exists a point $(z,y)\in K$ such
that $p_{\ast}(y) < z$. Then for any $t\in \mathbb{R}$, there exist
a neighborhood $U$ of $p_{\ast}(y)$ and a neighborhood $V$ of $z$
such that
$$u(t,U,y) < u(t,V,y).$$
\end{lemma}
%\begin{remark}\label{bounded-unbdd}
%Because of  items (iii)-(v) of above Theorem, we say that the fibers
%$A(\omega)$ share the common {\it ``bounded-or-unbounded property"}
%$uniformly for all the base point $\omega\in \Omega$.
%\end{remark}
\begin{proof}
By the minimality of $K$, for any $t\in \mathbb{R}$, there is
$\tau_n \rightarrow +\infty$ such that $\tau_n + t \geqq 0$ and
$$\Pi_{\tau_n} \circ \Pi_t (z,y) \rightarrow \Pi_t (z,y), ~\mbox{as}~n \rightarrow \infty.$$
Note that the monotonicity implies that
$$\Pi_{\tau_n} \circ \Pi_t (p_{\ast}(y),y)  \leqq  \Pi_{\tau_n} \circ \Pi_t (z,y).$$
Letting $n \rightarrow \infty$, we then get $\Pi_t(p_{\ast}(y),y)
\leqq \Pi_t(z,y)$, thus,
\begin{equation}\label{3.1}
u(t,p_{\ast}(y),y) \leqq u(t,z,y),~\forall t \in \mathbb{R}.
\end{equation}

Suppose that the conclusion of the lemma does not hold. Then we
claim that there exists $r_0\in \mathbb{R}$ such that
\begin{equation}\label{3.2}
u(t,p_{\ast}(y),y) = u(t,z,y),~\forall t \leqq r_0.
\end{equation}
Otherwise. By (\ref{3.1}), one has that for any $r \in \mathbb{R}$,
there exists some $\bar{t}\leqq r$ such that
$$u(\bar{t},p_{\ast}(y),y) < u( \bar{t},z,y).$$
Since $\Pi_t$ is strongly order-preserving, it follows that there
exist a neighborhood $\bar{U}$ of $u(\bar{t},p_{\ast}(y),y)$ and a
neighborhood $\bar{V}$ of $u(\bar{t},z,y)$ such that
$$u(r - \bar{t} + t_0,\bar{U},y\cdot \bar{t}) < u(r - \bar{t} + t_0,\bar{V},y\cdot \bar{t}).$$
Note that by the continuity of $\Pi_t$, there exist a neighborhood
$\hat{U}$ of $p_{\ast}(y)$ with $u(\bar{t},\hat{U},y)\subset
\bar{U}$, and a neighborhood $\hat{V}$ of $z$ with
$u(\bar{t},\hat{V},y)\subset \bar{V}$. So we have
\begin{equation*}
u(r - \bar{t} + t_0,u(\bar{t},\hat{U},y),y\cdot \bar{t}) < u(r -
\bar{t} + t_0,u(\bar{t},\hat{V},y),y\cdot \bar{t}).
\end{equation*}
Thus,
\begin{equation*}
u(r + t_0,\hat{U},y) < u(r + t_0,\hat{V},y).
\end{equation*}
Since $r$ is arbitrary, the conclusion of the lemma holds. A
contradiction. So we proved the claim.

By the minimality of $K$, we obtain that $\alpha(z,y) = K.$ Hence,
$(z,y)\in \alpha(z,y).$  Then it follows that there exists a
sequence $\tau_n \rightarrow -\infty$ such that $\tau_n \leqq r_0$
and $\Pi_{\tau_n} (z,y) \rightarrow  (z,y)$. Thus the 1-cover
property of $\omega(p_{\ast}(y),y)$ and (\ref{3.0}) imply that
$\Pi_{\tau_n} (p_{\ast}(y),y) \rightarrow (p_{\ast}(y),y)$. By
(\ref{3.2}), one has
$$u(\tau_n,p_{\ast}(y),y) = u(\tau_n,z,y).$$
By letting $n \rightarrow +\infty$, we get
$$(p_{\ast}(y),y)=(z,y).$$
A contradiction to the assumption. This completes the proof.
\end{proof}

\vskip 3mm

The following Proposition shows the 1-cover property of omega limit
sets for $\Pi_t$.

\begin{proposition}
\label{theorem 3.3} For any $(x_0,y_0)\in{X_+\times Y}$,
$\omega(x_0,y_0)$ is a 1-cover of $Y$.
\end{proposition}
\begin{proof}
Now fix  $(x_0,y_0)\in{X_+\times Y}$ and set $K = \omega(x_0,y_0)$.
For any $y\in Y$, by \cite[Proposition 3.1]{JZH}, we have
$(p_\ast(y),y)\leqq K\bigcap P^{-1}(y)$.

We claim that $\{(p_\ast(y),y)\}= K\bigcap P^{-1}(y), ~\forall y\in
Y$. Suppose not. Then there exist some $y\in Y$ and a point
$(\hat{z},y)\in K$ such that $p_{\ast}(y) < \hat{z}$. By the
minimality of $K$, we get that
\begin{equation*}
p_{\ast}(y) < z,~\forall(z,y)\in K\bigcap P^{-1}(y).
\end{equation*}
%Thus, $$p_{\ast}(y) < K \bigcap {P^{-1}(y)}.$$
Then it follows from Lemma \ref{lemma 3.1} that there exist a
neighborhood $U_z$ of $p_{\ast}(y)$ and a neighborhood $V_z$ of $z$
such that \begin{equation}\label{3.4}U_z < V_z.\end{equation} Since
$\{V_{z} : (z,y) \in K \bigcap {P^{-1}(y)}\}$ is an open cover of $K
\bigcap {P^{-1}(y)}$, we can find a finite subcover, denoted by
$\{V_1,V_2, \cdots, V_n\}$. Note that by (\ref{3.4}) there exist
neighborhoods $U_i,~i=1,2,\cdots,n$ of $p_{\ast}(y)$ such that
$$U_1 < V_1, ~U_2 < V_2,~\cdots, ~U_n < V_n.$$
Therefore, $\bigcap_{i=1}^{n}{U_i} < \bigcup_{i=1}^{n}{V_i}.$ Since
$K \bigcap {P^{-1}(y)}\subset \bigcup_{i=1}^{n}{V_i}$, we have
\begin{equation*}
\bigcap_{i=1}^{n}{U_i} < K \bigcap {P^{-1}(y)}.
\end{equation*}
So we can take an $\epsilon_0 > 0$ such that
\begin{equation}\label{3.5}
B^+(p_{\ast}(y),\epsilon_0) < K \bigcap {P^{-1}(y)},
\end{equation}
where $B^+(p_{\ast}(y),\epsilon_0) = \{ x \in X_+ : x \geqq
p_{\ast}(y),~\norm{x - p_{\ast}(y)} \leqq \epsilon_0\}$. By the
uniform stability of $\Pi_t(p_{\ast}(y),y)$, there exists $\delta_0
= \delta_0(\epsilon_0)\leqq \epsilon_0$ such that
%\begin{equation}\label{3.6}
%\norm{u(t,x,y) - u(t,p_{\ast}(y),y)} \leqq \epsilon_0,\quad\forall
%t\geqq 0
%\end{equation}
\begin{equation*}
\norm{ u - p_{\ast}(y) } \leqq \epsilon_0,~\forall (u,y) \in
\omega(x,y)\bigcap{P^{-1}(y)}
\end{equation*}
whenever $\norm{x - p_{\ast}(y)} \leqq \delta_0$. Combing with
(\ref{3.5}), we get
$$(p_{\ast}(y),y) \leqq \omega(x,y)  \bigcap{P^{-1}(y)} < K \bigcap
{P^{-1}(y)}$$ for any $x\in B^+(p_{\ast}(y),\delta_0) $.
%Thus, by
%the definition of $p(y)$, it follows that
%$$(p_{\ast}(y),y) \leqq \omega(x,y)  \bigcap{P^{-1}(y)} \leqq (p(y),y).$$
Since $\omega(x,y)$ is minimal, using \cite[Proposition 3.1
(3)]{JZH}, we obtain \begin{equation}\label{3.6}\omega(x,y) =
\omega(p(y),y)=p_\ast(\cdot),\quad\forall x \in
B^+(p_{\ast}(y),\delta_0).\end{equation}

%Let $L \subset P^{-1}(y)$ denote the segment with endpoints
%$(p_{\ast}(y),y)$ and $(z,y) \in K$, and
Set
$$L = \{\tau \in [0,1]: x_\tau = p_{\ast}(y) + \tau(\hat{z} - p_{\ast}(y)),~\omega(x_\tau,y) = p_\ast(\cdot) \}.$$
By (\ref{3.6}), there exists a $\bar{\tau}
> 0$ such that $[0,\bar{\tau}] \subset L.$ It's easy to see that $L$ is an
interval. Now we show that $L$ is closed, that is, $L = [0,\tau_0]$
with $0 < \tau_0=\sup\{\tau:\tau\in L\} <1$. Note that
$\Pi_t(x_{\tau_0},y)$ is uniformly stable. Let $\delta(\epsilon)$ be
the modulus of uniform stability for $\epsilon>0$. Thus, we take
$\tau\in[0,\tau_0)$ with $\norm{x_\tau-x_{\tau_0}}<\delta(\epsilon)$
and we get
$$\norm{u(t,x_\tau,y)-u(t,x_{\tau_0},y)}<\epsilon,~~\forall
t\geqq0.$$ Since $\omega(x_\tau,y)=p_\ast(\cdot)$, there is a
$\hat{t}$ such that $$\norm{u(t,x_\tau,y)-p_\ast(y\cdot
t)}<\epsilon,~~\forall t\geqq \hat{t}.$$ Then, we deduce that
$$\norm{u(t,x_{\tau_0},y)-p_\ast(y\cdot t)}<2\epsilon,~~\forall
t\geqq \hat{t},$$ and hence $\omega(x_{\tau_0},y)=p_\ast(\cdot)$. So
$L$ is closed.

Then by a similar argument in the proof of \cite[Theorem 4.1]{JZH},
we can get a contradiction. Indeed, since $L = [0,\tau_0]$ with $0 <
\tau_0 <1$, for any $\tau \in (\tau_0,1)$ we have $(p_{\ast}(y),y)
\notin \omega(x_\tau,y)$. For $\epsilon_0$ defined in (\ref{3.5}),
by the uniform stability of the orbit, we get
\begin{equation}\label{3.7}
\norm{ u(t,x_\tau,y) - u(t,x_{\tau_0},y) } < \epsilon_0,
\quad\forall t \geqq 0
\end{equation}
whenever $0 < \tau - \tau_0 \ll 1 $. Let $\{t_n\}$ be such that
$\Pi_{t_n}(x_{\tau_0},y) \rightarrow (p_{\ast}(y),y) $. Choosing a
subsequence if necessary, we may assume that $\Pi_{t_n}(x_{\tau},y)
\rightarrow (\tilde {x},y)$ for $0<\tau-\tau_0\ll 1$. By
(\ref{3.7}), we obtain $\norm{ \tilde {x} - p_{\ast}(y) } \leqq
\epsilon_0$. Thus, from the monotonicity, $\tilde {x} \in
B^+(p_{\ast}(y),\epsilon_0)$. So by (\ref{3.5}), $\tilde {x} <
K\bigcap {P^{-1}(y)}$. Using \cite[Proposition 3.1 (3)]{JZH} again,
we get $\omega(\tilde {x},y) = \omega(p(y),y)=p_\ast(\cdot)$. Then
the minimality of $\omega(x_{\tau},y)$ implies that
$\omega(x_{\tau},y) = \omega(\tilde {x},y) = p_\ast(\cdot)$, which
is a  contradiction to the definition of $\tau_0$. Thus, $K \bigcap
{P^{-1}(y)} = \{(p_{\ast}(y),y)\}$ for all $y\in Y$. The minimality
deduces that $K$ is a 1-cover of $Y$.
\end{proof}

Denote $$A = \bigcup_{c(\cdot) \textnormal{ is a 1-cover for
$\Pi_t$}} c(\cdot)
$$ of all 1-covers of $Y$ for $\Pi_t$. For each $y \in Y$, set $A(y)
= A\bigcap P^{-1}(y)$. Based on Proposition \ref{theorem 3.3}, we
can get

\begin{lemma}
\label{lemma 4.1} $A$ is totally ordered with respect to `$<$', and
for each $y \in Y$, $A(y)$ is homeomorphic to a closed interval in
$\mathbb{R}$.
\end{lemma}

\begin{proof}
The proof is similar to that of Theorem 3.1 in \cite{CGW}.

\end{proof}

For any $(x_0,y_0)\in X_+\times Y$, denote the forward orbit and the
omega-limit set for $\Gamma_t$ by $O^+_{\Gamma}(x_0,y_0)$ and
$\omega_\Gamma(x_0,y_0)$, respectively. Now we will prove the
$1$-cover property for the uniformly stable $\omega$-limit sets of
the comparable skew-product semiflow $\Gamma_t$.

\vskip 3mm

\begin{proposition}\label{prop 4.3}
Assume that for point $(x_0,y_0)\in X_+\times Y$,
$O^+_{\Gamma}(x_0,y_0)$ is uniformly stable. Let
$\hat{K}=\omega_\Gamma(x_0,y_0)$. For any $y \in Y$, if there exists
some $ (b(y),y) \in A(y)$ such that $\hat{K} \bigcap P^{-1}(y) \geqq
(b(y),y)$, then $\hat{K}$ is a 1-cover of $Y$ for $\Gamma_t$.
\end{proposition}

\begin{proof}

Let $C_\Pi=\{c(\cdot): c(\cdot) \textnormal{ is a $1$-cover for
}\Pi_t\}.$ Then by a similar argument in the proof of \cite[Theorem
4.3]{CGW}, using Lemma \ref{lemma 4.1} we can define a nonempty
totally ordered set $\mathcal{C}\subset C_\Pi$, for which
\begin{equation*}
\mathcal{C}=\{c(\cdot)\in C_\Pi :(c(y),y) \geqq \hat{K}\cap
P^{-1}(y) \,\mbox{ for all } \,y\in Y\},
\end{equation*}  and the greatest lower bound $\inf\mathcal{C}\in \mathcal{C}$ exists.

Denote $q(\cdot)=\inf\mathcal{C}$. Now we assert that $\hat{K}$ is a
1-cover of $Y$ for $\Gamma_t$, satisfying
\begin{equation*} \hat{K} \bigcap P^{-1}(y) = (q(y),y),~\forall
y \in Y.\end{equation*} Otherwise, there exist a $y_1 \in Y$ and
some $(c,y_1) \in \hat{K}\bigcap P^{-1}(y_1)$ such that
\begin{equation*}
(q(y_1),y_1) > (c,y_1).
\end{equation*}
According to our assumption, we have
$$(q(y_1),y_1) > (c,y_1) \geqq (b(y_1),y_1). $$
Then by \cite[Lemma 3.4]{CGW}, there is a strictly order-preserving
continuous path
\begin{equation}\label{4.1}
J: [0,1] \rightarrow A(y_1) \mbox{ with }J(0) = (b(y_1),y_1)\mbox{
and }J(1) = (q(y_1),y_1).\end{equation} Since $(q(y_1),y_1) >
(c,y_1)$, by the strongly order-preserving property of $\Pi_t$ and
the comparability of $\Gamma_t$ with respect to $\Pi_t$, we have
that there exists a neighborhood $U$ of $q(y_1)$ such that
\begin{equation*}
\Pi_{t_1}(U,y_1) > \Pi_{t_1}(c,y_1) \geqq
\Gamma_{t_1}(c,y_1)=(v(t_1,c,y_1),y_1\cdot t_1)
\end{equation*} for some $t_1 > t_0$.
Denote $\bar{c}=v(t_1,c,y_1)$ and $y_2=y_1\cdot t_1$. Then
$(\bar{c},y_2)\in \hat{K}$ and
\begin{equation}\label{4.2}
(u(t_1,U,y_1),y_2)>(\bar{c},y_2).
\end{equation}
Note that $U$ is a neighborhood of $q(y_1)$. Then due to (\ref{4.1})
we can find a point $q_1(y_1) \in U\bigcap A(y_1)$ with
$q_1(y_1)<q(y_1)$. Thus, by (\ref{4.2}) we obtain
\begin{equation*}
(q(y_2),y_2) > (q_1(y_2),y_2) > (\bar{c},y_2).
\end{equation*}

Since $O^+_{\Gamma}(x_0,y_0)$ is uniformly stable, by Theorem
\ref{theorem 3.0} $\hat{K}$ admits a flow extension which is
minimal. Thus for any $t \in \mathbb{R}$, there is $t_n \to +\infty$
such that $t_n + t \geqq 0$ and
$$\Gamma_{t_n} \circ \Gamma_t(\bar{c},y_2) \to \Gamma_t(\bar{c},y_2),~n \to \infty.$$
Then the monotonicity and the comparability of $\Gamma_t$ with
respect to $\Pi_t$ imply that
$$\Pi_{t_n} \circ \Pi_t(q_1(y_2),y_2) \geqq \Pi_{t_n} \circ \Pi_t(\bar{c},y_2)\geqq \Gamma_{t_n} \circ \Gamma_t(\bar{c},y_2).$$
By letting $n \to \infty$ in the above, we get $\Pi_t(q_1(y_2),y_2)
\geqq \Gamma_t(\bar{c},y_2)$, thus,
\begin{equation}\label{4.3}
u(t,q_1(y_2),y_2) \geqq v(t,\bar{c},y_2),~\forall t \in \mathbb{R}.
\end{equation}
Note that $O^+_{\Pi}(q_1(y_2),y_2)$ is uniformly stable, by Theorem
\ref{theorem 3.0} we get that
\begin{equation}\label{4.4}
u(t,q_1(y),y)=q_1(y\cdot t)~\mbox{ for any }y\in Y \mbox{ and }
t\in\mathbb{R}.
\end{equation}
So combining (\ref{4.3}), (\ref{4.4}) and the comparability of
$\Gamma_t$ with respect to $\Pi_t$, similarly as the proof of Lemma
\ref{lemma 3.1}, we can get that for any $t \in \mathbb{R}$, there
exist a neighborhood $U_t$ of $q_1(y_2)$ and a neighborhood $V_t$ of
$\bar{c}$ such that
\begin{equation*}u(t,U_t,y_2) > v(t,V_t,y_2).\end{equation*}
In particular, for $t = 0$, there exist a neighborhood $U_0$ of
$q_1(y_2)$ and a neighborhood $V_0$ of $\bar{c}$ such that
\begin{equation}\label{4.5}
(U_0,y_2) > (V_0,y_2).
\end{equation}

Recall that $\hat{K}$ is the omega-limit set of $(x_0,y_0)$ for
$\Gamma_t$, there exists some sequence $t_n \to +\infty$ such that
$\Gamma_{t_n}(x_0,y_0) \to (\bar{c},y_2) \in \hat{K}$, as $n \to
\infty$. Also, since $q_1(\cdot)$ is a 1-cover for $\Pi_t$, we get
$\Pi_{t_n}(q_1(y_0),y_0) \to (q_1(y_2),y_2)$, as $n \to \infty$. So
by (\ref{4.5}) there exists $N > 1$ such that
\begin{equation}\label{20}
\Pi_{t_N}(q_1(y_0),y_0) > \Gamma_{t_N}(x_0,y_0).\end{equation} Then
by a similar argument in the proof of \cite[Theorem 4.3]{CGW}, we
can get that
\begin{equation*} (q_1(y),y) \geqq \hat{K} \bigcap
P^{-1}(y) ~~\mbox{ for all }y \in Y.
\end{equation*}
For the sake of completeness, we include a detailed proof here. As a
matter of fact, by the monotonicity of $\Pi_t$ and the comparability
of $\Gamma_t$ with respect to $\Pi_t$, it follows from (\ref{20})
that
\begin{equation}\label{4.6}
\Pi_{t+t_N}(q_1(y_0),y_0) \geqq \Pi_t\Gamma_{t_N}(x_0,y_0) \geqq
\Gamma_{t+t_N}(x_0,y_0),~\forall t \geqq 0.
\end{equation}
For any $(x,y) \in \hat{K}$, there exists $s_n \to +\infty$ such
that $\Gamma_{s_n}(x_0,y_0) \to (x,y)$, as $n \to \infty$. Let $t =
s_n - t_N$ in (\ref{4.6}) for all $n$ sufficiently large. Then we
get $\Pi_{s_n}(q_1(y_0),y_0) \geqq \Gamma_{s_n}(x_0,y_0)$. Letting
$n \to +\infty$, one has $(q_1(y),y) \leqq (x,y)$. By the
arbitrariness of $(x,y) \in \hat{K}$, we get $(q_1(y),y) \geqq
\hat{K} \bigcap P^{-1}(y)$ for all $y \in Y$. This contradicts the
definition of $q(\cdot)$. So we have proved the assertion, and
$\hat{K}$ is a 1-cover of $Y$ for $\Gamma_t$.

\end{proof}

\begin{proof}[{\bf Proof of Theorem \ref{theorem 2.1}}]
Let $v(t,\cdot,v_0;d,F)$ be an $L^\infty$-bounded solution of
\eqref{IBVP-sys} in $X_+$. Then from the study in \cite{Hen} and a
priori estimates for parabolic equations, it follows that $v$ is a
globally defined classical solution in $X_+$, and
$\{v(t,\cdot,v_0;d,F):t\geqq \tau\}$ is precompact in $X_+$ for some
$\tau>0$. So $\hat{K}:=\omega_\Gamma(v_0,(d,F))$ is a nonempty
compact set in $X_+\times H(d,F)$. Since $0(\cdot)\in C_\Pi$ by
(H2), $$\hat{K}\cap P^{-1}(y)\geqq (0,y)\in A(y),~~ \forall y\in
Y.$$  If $v(t,\cdot,v_0;d,F)$ is uniformly stable, then by
Proposition \ref{prop 4.3} we get that $\hat{K}$ is a $1$-cover of
$\Omega$ for $\Gamma_t$, and thus the uniformly stable
$L^\infty$-bounded solution $v(t,\cdot,v_0;d,F)$ is asymptotic to an
almost periodic solution.
\end{proof}

\bibliographystyle{amsplain}

\end{document}